\documentclass[11pt,a4paper]{amsart}
\usepackage{preamble}
\RequirePackage[type1]{crimson}

\begin{document}
\title{Deformations of Reducible Galois Representations to Hida-families}
\author[Anwesh Ray]{Anwesh Ray}
\address{Department of Mathematics \\ University of British Columbia \\
  Vancouver BC, V6T 1Z2, Canada.} 
  \email{anweshray@math.ubc.ca}

\begin{abstract}The global deformation theory of residually reducible Galois representations with fixed auxiliary conditions is studied. We show that $\bar{\rho}:\op{Gal}(\bar{\Q}/\Q)\rightarrow \op{GL}_2(\bar{\F}_p)$ lifts to a Hida line for which the weights range over a congruence class modulo-$p^2$. The advantage of the purely Galois theoretic approach is that it allows us to construct $p$-adic families of Galois representations lifting the actual representation $\bar{\rho}$, and not just the semisimplification.

\end{abstract}

\maketitle

\section{Introduction}
Let $\F_q$ be the finite field of characteristic $p>2$ and denote by $\mathcal{O}$ the ring of Witt vectors $\op{W}(\F_q)$. Fix a continuous Galois representation $\bar{\rho}:\op{Gal}(\bar{\Q}/\Q)\rightarrow \GL_2(\F_q)$. In this paper, we consider the problem of describing various families of $p$-adic representations which lift $\bar{\rho}$. These representations will be subject to certain local conditions at the primes at which they are allowed to ramify.
\par We first discuss the case when $\bar{\rho}$ is absolutely irreducible. Letting $S$ be a finite set of primes containing $p$ and the primes at which $\bar{\rho}$ is ramified, set $\operatorname{G}_{\Q,S}:=\operatorname{Gal}(\Q_S/\Q)$, where $\Q_S$ is the maximal algebraic extension of $\Q$ such that all primes $v\notin S$ are unramified. At each prime $v\notin S$, set $\op{G}_v:=\op{Gal}(\bar{\Q}_v/\Q_v)$ and fix an embedding $\bar{\Q}\hookrightarrow \bar{\Q}_v$. The functor of deformations of $\bar{\rho}$ was introduced by Mazur \cite{mazur}, following Hida \cite{hidagl2}. Mazur shows that there is a \textit{universal deformation ring} $\op{R}(\bar{\rho})$ and a \textit{universal deformation} $\rho^{\op{univ}}: \G_{\Q,S}\rightarrow \GL_2(\op{R}(\bar{\rho}))$ which represents the functor of deformations of $\bar{\rho}$ that are unramified away from the set $S$. Boston and Mazur in \cite{boston1},\cite{boston} and \cite{bostonmazur} calculate $\op{R}(\bar{\rho})$ in many cases of natural interest. In \cite{mazurinfinitefern}, Mazur studies the structure of an "infinite fern" in the universal deformation space of Galois representations. The study of deformation functors has to some extent been motivated by Serre's conjecture. When $\op{det}\bar{\rho}$ is odd, Serre's conjecture asserts that $\bar{\rho}$ lifts to a characteristic zero representation $\rho$ attached to a Hecke eigencuspform. The conjecture was settled by Khare-Wintenberger in \cite{KW2}. Prior to this development, Ramakrishna in \cite{RamLGR} and \cite{RamakrishnaFM} showed via purely Galois theoretic means that if $\bar{\rho}$ satisfies some favorable conditions, it exhibits a characteristic zero lift which is \textit{geometric} in the sense of Fontaine-Mazur \cite[p. 193]{fontainemazur}. The construction involves a purely Galois theoretic argument, involving the study of global deformations satisfying fixed local deformation conditions at the primes $S$ and a finite auxiliary set of primes $X$. The auxiliary primes in $X$ are called \textit{nice} primes and the deformation functor is of Steinberg type \cite[section 4.2]{PatEx}. When $\bar{\rho}$ is \textit{ordinary}, i.e, $\bar{\rho}( \op{G}_p)$ is contained in $\mtx{\ast}{\ast}{0}{\ast}\subset \op{GL}_2(\F_q)$, the deformation condition at $p$ consists of ordinary deformations. The functor of global Galois deformations unramified outside $S\cup X$ with fixed local conditions is represented by a \textit{universal deformation ring} $\mathcal{R}_{S\cup X}$. Denote by \[\rho_{S\cup X}^{\op{univ}}:\op{G}_{\Q,S\cup X}\rightarrow \GL_2(\mathcal{R}_{S\cup X})\] the universal deformation (where the determinant character is not fixed). Let $\Q^{\op{cyc}}$ denote the cyclotomic $\Z_p$-extension of $\Q$ and set $\Gamma:=\Gal(\Q^{\op{cyc}}/\Q)$. The Iwasawa-algebra $\Lambda$ is the completed group algebra $\Lambda:=\mathcal{O}[[\Gamma]]$. Fixing a topological generator $\gamma\in \Gamma$ amounts to fixing an isomorphism of $\mathcal{O}$-algebras $\Lambda\simeq \mathcal{O}[[x]]$, where $\gamma-1$ is the variable $x$. Refer to $\op{Spec}\Lambda$ as \textit{weight-space} and the map to weight-space $\mathcal{W}t:\op{Spec} \mathcal{R}_{S\cup X}\rightarrow \op{Spec}\Lambda$ is induced by mapping $x$ to $\det\rho_{S\cup X}^{\op{univ}}(\gamma)-1$. For every fixed lift $\kappa:\op{G}_{\Q,S}\rightarrow \GL_1(\mathcal{O})$ of $\det\bar{\rho}$, let $\mathcal{R}_{S\cup X}^{\kappa}$ denote the \textit{fixed weight} quotient of $\mathcal{R}_{S\cup X}$, parametrizing deformations $\rho$ for which $\det \rho=\kappa$. Ramakrishna shows that $X$ may be chosen so that $\mathcal{R}_{S\cup X}^{\kappa}\simeq \mathcal{O}$ for each $\kappa$, cf. \cite[Theorem 1]{RamakrishnaFM}. It follows that there is a canonical isomorphism $\mathcal{R}_{S\cup X}\xrightarrow{\sim} \mathcal{O}[[x]]$ induced by the map to weight-space. This produces a nice family of ordinary Galois representations, one for each $p$-adic weight. Such a family of $p$-adic Galois representations is represented by a $p$-adic line which is a component in a certain $p$-ordinary local Hecke algebra of specified level (depending on the set of primes $S\cup X$) and arbitrary weight. We refer to such a component as a \textit{Hida-line}. Presentations of deformation rings with prescribed local conditions are studied by B\"ockle in \cite[section 5]{Bockle}. It is necessary that the local deformation conditions be representable. In this setting, it becomes possible to provide a description for generators and relations in terms of certain Galois cohomology groups with local conditions. The local conditions are associated with the tangent spaces to the various local deformation functors which are to be interpreted as \textit{Selmer conditions}. Even in cases when the structure of the minimal deformation ring $\mathcal{R}_S$ may be complicated, allowing for additional ramification and imposing local conditions gives a smooth ring $\mathcal{O}[[x]]$. The existence of such families of Galois representations are of genuine interest from the point of view of studying the structure of universal deformation rings.
\par The goal of this paper is to study the analogous deformation problem with prescribed local conditions in the setting when $\bar{\rho}$ is reducible, yet indecomposable. The reducible case presents some technical difficulties, which we explain. Hamblen and Ramakrishna in \cite{hamblenramakrishna} show that under some favorable conditions on
\[\bar{\rho}=\mtx{\varphi}{\ast}{0}{1}:\operatorname{G}_{\Q,S}\rightarrow \op{GL}_2(\F_q),\] there is a finite set of auxiliary primes $X$ and a $p$-ordinary lift $\rho:\op{G}_{\Q,S\cup X}\rightarrow \GL_2(\mathcal{O})$ which satisfies certain deformation problems at the primes $v\in S\cup X$. By the result of Skinner and Wiles in \cite{skinnerwiles}, the representation $\rho$ arises from a $p$-ordinary Hecke eigencuspform. The primes $v\in X$ are called \textit{trivial primes} since the restriction $\bar{\rho}_{\restriction \op{G}_{v}}$ is the trivial representation \cite[section 4]{hamblenramakrishna}. Each trivial prime $v$ is equipped with a deformation functor $\tilde{\mathcal{C}}_v$ (see \cite[Definition 26 and 30]{hamblenramakrishna} for the definition in the setting when the determinant of the lifts is fixed). The deformation functor at $v$ is \textit{not} representable, as we explain. Let $\op{Ad}\bar{\rho}$ be the $\F_q$-space of $2\times 2$-matrices with adjoint Galois action. For $g\in \op{Gal}(\bar{\Q}/\Q)$ and $v\in \op{Ad}\bar{\rho}$, we have that $g\cdot v:=\bar{\rho}(g)v \bar{\rho}(g)^{-1}$. If a functor of deformations $\mathcal{F}$ of $\bar{\rho}_{\restriction \op{G}_{v}}$ is representable, it comes equipped with a \textit{tangent space} $\mathcal{N}_{\mathcal{F}}:=\mathcal{F}(\F_q[\epsilon]/(\epsilon^2))$, which may be identified with a subspace of $H^1(\op{G}_{v},\op{Ad}\bar{\rho})$. For $\varrho\in \mathcal{F}(\mathcal{O}/p^N)$ and $f\in \mathcal{N}_{\mathcal{F}}$, \[(\op{Id}+f p^{N-1})\varrho\in \mathcal{F}(\mathcal{O}/p^N).\] In other words, the tangent-space must stabilize the functor of deformations. This is not the case for the deformation problem at a trivial prime $v$. Associated to the deformation functor $\tilde{\mathcal{C}}_v$ at $v$, there is a space of cohomology classes which plays the role of a tangent space for mod $p^3$ lifts of $\bar{\rho}_{\restriction \op{G}_{v}}$. These stabilize $\tilde{\mathcal{C}}_v(\mathcal{O}/p^N)$ for $N\geq 3$, but not for $N\leq 2$ (cf. \cite[Corollary 25]{hamblenramakrishna} and Proposition $\ref{furtherdetailed}$). It has thus not been possible to formulate what the analogous deformation rings should be in the reducible setting (when local conditions at trivial primes are taken into account) and whether such deformation rings are well-defined. For each choice of lift $\kappa$ of $\det\bar{\rho}$, the method of Hamblen-Ramakrishna relies on a choice of a mod-$p^2$ lift $\rho_2$ of $\bar{\rho}$ (cf. \cite[section 5]{hamblenramakrishna}) which depends on $\kappa$. The lift $\rho_2$ is ramified at two trivial primes $\{v_1,v_2\}$ (in addition to $S$) which depend on the choice of $\kappa$ (cf. \cite[Theorem 41]{hamblenramakrishna}). Thus, when local conditions are prescribed, one does not expect all the lifts obtained from the method Hamblen-Ramakrishna to lie on the same component in a single Hida family, let alone a single Hida-line.
\par We examine an adaptation of the functor of global deformations, with two crucial differences:
\begin{enumerate}
    \item We restrict the functor to a suitable subcategory $\mathfrak{C}$ of the category of coefficient rings.
    \item Second, we fix a mod $p^2$ lift $\rho_2$ of $\bar{\rho}$ and study the functor of deformations of $\rho_2$ (restricted to $\mathfrak{C}$) with prescribed local conditions.
\end{enumerate} 
A coefficient ring $R$ over $\mathcal{O}$ is a Noetherian local $\mathcal{O}$-algebra $R$ with maximal ideal $\mathfrak{m}_R$ and residue field isomorphic to $\F_q$. Maps between coefficient rings are $\mathcal{O}$-algebra maps of local rings. The category $\mathfrak{C}$ consists of finite-length coefficient rings for which $p\notin \mathfrak{m}_R^2$. Suitable local deformation conditions $\tilde{\mathcal{C}}_v$ at each prime $v\in S$ are chosen. These deformation conditions will be of Steinberg-type for $v\in S\backslash \{p\}$ and $\tilde{\mathcal{C}}_p$ will consist of ordinary deformations. By the construction of Hamblen-Ramakrishna, there is a suitable choice of auxiliary primes $X$ disjoint from $S$ which allow for a characteristic zero lift
\[\rho:\op{G}_{\Q,S\cup X}\rightarrow \GL_2(\mathcal{O})\]which satisfies the following conditions:
\begin{enumerate}
    \item $\rho$ satisfies the conditions $\tilde{\mathcal{C}}_v$ for $v\in S\cup X$,
    \item $\rho$ arises from a Hecke eigencuspform.
\end{enumerate} Set
$\Phi$ to denote the tuple of deformation conditions $(\tilde{\mathcal{C}}_v)_{v\in S\cup X}$ and denote by $\rho_2:\G_{\Q,S\cup X}\rightarrow \GL_2(\mathcal{O}/p^2)$ the mod $p^2$ reduction of $\rho$. The functor of deformations we consider depends on $\Phi$ and the choice of $\rho_2$. For $R\in \mathfrak{C}$, let $\rho_{2,R}$ denote the deformation with image in $\GL_2(R/(pR\cap \mathfrak{m}_R^2))$ induced from $\rho_2$ by the structure map $\mathcal{O}\rightarrow R$. Let \[\Df:\mathfrak{C}\rightarrow \operatorname{Sets}\] be the functor such that $\Df(R)$ consists of deformations $\rho_R:\G_{\Q,S\cup X}\rightarrow \GL_2(R)
$ for which
\[\rho_{2,R}= \rho_R\mod{(pR\cap \mathfrak{m}_R^2)}.\]
The conditions $\eqref{c1ofmain}$ to $\eqref{c10ofmain}$ below are those imposed on $\bar{\rho}$ by the method of Hamblen-Ramakrishna.
\begin{Th}\label{main}
Let $S$ be a finite set of primes containing $p$ and $\bar{\rho}:\G_{\Q,S}\rightarrow \op{GL}_2(\F_q)$ be a two-dimensional Galois representation given by $\bar{\rho}=\mtx{\varphi}{\ast}{0}{1}$. Let $c$ denote complex-conjugation. Suppose further that 
\begin{enumerate}
\item\label{c1ofmain} the characteristic $p\geq 3$,
\item the representation $\bar{\rho}$ is indecomposable,
\item $\bar{\rho}$ is odd, i.e. $det\bar{\rho}(c)=-1$, where $c$ denotes complex conjugation,
\item\label{c8ofmain}  the character $\varphi_{\restriction I_p}=\chi^{k-1}_{\restriction I_p}$ where $2\leq k\leq p-1$,
\item\label{c9ofmain}
$\varphi_{\restriction \G_p}\notin\{ \bar{\chi}_{\restriction \G_p}, \bar{\chi}^{-1}_{\restriction \G_p},1\}$,
\item\label{c10ofmain} the $\F_p$-span of the image of $\varphi$ is (the entirety of) $\F_q$.
\end{enumerate}

There exists a deformation $\tilde{\varrho}$ as depicted:
 \[\begin{tikzpicture}[node distance = 2.0cm, auto]
      \node (GSX) {$\G_{\Q,S\cup X}$};
      \node (GS) [right of=GSX] {$\G_{\Q,S}$};
      \node (GL2) [right of=GS]{$\op{GL}_2(\F_q).$};
      \node (GL2W) [above of= GL2]{$\op{GL}_2(\mathcal{O}[[U]])$};
      \draw[->] (GSX) to node {} (GS);
      \draw[->] (GS) to node {$\bar{\rho}$} (GL2);
      \draw[->] (GL2W) to node {} (GL2);
      \draw[dashed,->] (GSX) to node {$\tilde{\varrho}$} (GL2W);
      \end{tikzpicture}\] such that for $R\in \mathfrak{C}$, the induced map
\begin{equation}\label{surjective}\tilde{\varrho}^*:\Hom(\mathcal{O}[[U]],R)\rightarrow \D(R)\end{equation} is surjective. Furthermore, if $R$ contains no $p$-torsion, the above map $\eqref{surjective}$ is an isomorphism.
\end{Th}
The lift $\tilde{\varrho}$ does not represent the functor $\D$, however, the condition on the surjectivity of the map $\eqref{surjective}$ shows that all deformations satisfying $\D$ may be recovered from $\tilde{\varrho}$. For instance, all modular lifts of $\bar{\rho}$ satisfying $\D$ are characteristic zero points on the Hida line defined by $\tilde{\varrho}$. It is easy to see that $\D(\F_q[\epsilon]/(\epsilon^2))=\{\bar{\rho}\}$ and hence the map $\eqref{surjective}$ is not an isomorphism (for $R=\F_q[\epsilon]/(\epsilon^2)$). The assertion that $\eqref{surjective}$ is an isomorphism when $R$ has no $p$-torsion indicates that $\tilde{\varrho}$ has many of the characteristic properties of a universal deformation.
\par Next, we describe the map to weight-space for the above Galois representation. Associate to a Galois representation $\rho_R:\G_{\Q}\rightarrow \GL_2(R)$, the \textit{weight}, which is a point on weight-space $\op{Spec}\Lambda$,
\[\mathscr{W}t(\rho_R):\Spec R\rightarrow \Spec \Lambda\] induced by the homomorphism of rings mapping $x$ to $\det\rho_R(\gamma)-1$.
\begin{Th}\label{aux}
Let $R$ be a coefficient ring in $\mathfrak{C}$ with maximal ideal $\mathfrak{m}_R$. In the context of Theorem $\ref{main}$, the functor of points of the map to weight space induces on $R$-points a map \[\mathscr{W}t^*:\D(R)\rightarrow \Hom_{\mathfrak{C}}(\Lambda,R)\]whose image consists weights in the mod $(pR)\cap \mathfrak{m}_R^2$ congruence class of weights which coincide with the weight of the chosen lift $\rho_2$ modulo $(pR)\cap \mathfrak{m}_R^2$.
\end{Th}
\begin{Remark}
Theorem $\ref{aux}$ implies in particular that on $\mathcal{O}$-points, the image of map 
\[\mathscr{W}t^*:\D(\mathcal{O})\rightarrow \Hom_{\mathfrak{C}}(\Lambda,\mathcal{O})\] is a congruence class of weights modulo $p^2$.
\end{Remark}
\par It should be noted here that an alternative approach to studying the deformation theory of residually reducible Galois representations is via pseudo-representations. The study of pseudorepresentations also aids in proving modularity of a characteristic zero geometric lift of $\bar{\rho}$ when it is known to exist (see for instance \cite{skinnerwiles}). In this paper, we study the deformations of the actual representation $\bar{\rho}$ and not just its semisimplification. In practice, examples of reducible mod-$p$ representations $\bar{\rho}$ are constructed via class field theory. A simple application of inflation-restriction shows that a Galois stable line in a mod-$p$ ray class group of $\Q(\mu_p)$ gives rise to a residually reducible Galois representation $\bar{\rho}$ (see \cite{rayconstructing}). Depending on the set $S$, there may exist many (non-isomorphic) representations $\bar{\rho}$ with the same semisimplification. Every one of the representations $\bar{\rho}$ gives rise to a unique Hida-line. This is the advantage of the purely Galois theoretic approach used in this manuscript.
\par The paper is organized as follows: in section $\ref{section2}$, local deformation conditions at the primes $v\in S$ are discussed. Dimensions of local tangent spaces are calculated and Selmer groups associated to these local deformation conditions are defined. Section $\ref{highlyversal}$ introduces the special role played by \textit{trivial} primes, i.e. the auxiliary primes at which deformations are allowed to further ramify. In section $\ref{section4}$ the main results of this paper are proved.
\subsection*{Acknowledgements}
The author is very grateful to his advisor Ravi Ramakrishna for introducing him to the fascinating subject of Galois deformations. He is very thankful to the anonymous referee, for helpful suggestions which have led to significant improvement in the quality of exposition and presentation of the mathematical content in this paper.
\section{The Setup}\label{section2}
In this section we review some preliminaries. We study global deformation conditions with local constraints. The reader may refer to \cite{Mazurintro} for an introduction to the Galois deformation theory. In the residually irreducible case, presentations of Galois deformation rings with local constraints are discussed in detail in \cite{Bockle}.
\subsection{Deformation conditions and tangent spaces}
\par Suppose that $\bar{\rho}: \G_{\Q,S}\rightarrow \text{GL}_2(\F_q)$ satisfies the conditions of Theorem $\ref{main}$. Recall that $k$ is chosen such that\[\bar{\rho}_{\restriction I_p}=\mtx{\chi^{k-1}}{\ast}{0}{1}.\]By assumption, $2\leq  k\leq  p-1$. In this section, we describe some of the local deformation conditions at primes $v\in S$. Denote by $\mathcal{C}_{\mathcal{O}}$ (resp. $\mathcal{C}_{\mathcal{O}}^f$) the category of noetherian (resp. finite length) coefficient-rings over $\mathcal{O}$. Let $R$ be a coefficient ring with residue map $R\rightarrow \F_q$. Set $\widehat{\op{GL}}_2(R)$ to denote the kernel of the map
\[\op{GL}_2(R)\rightarrow \op{GL}_2(\F_q).\]We recall the definition of a Galois deformation.
\begin{Def}\label{stricteq} Let $R\in \mathcal{C}_{\mathcal{O}}$, $\Pi$ denote $\op{G}_{\Q}$ (resp. $\op{G}_{v}$) and $\bar{r}$ denote $\bar{\rho}$ (resp. $\bar{\rho}_{\restriction \op{G}_{v}}$). Two lifts $\rho,\rho':\Pi\rightarrow \op{GL}_2(R)$ of $\bar{r}$ are \textit{strictly-equivalent} if $\rho=A\rho' A^{-1}$ for $A\in \widehat{\op{GL}}_2(R)$. A \textit{deformation} is a strict equivalence class of lifts.
      \end{Def}
      
      Recall that $\mathfrak{C}\subset \mathcal{C}_{\mathcal{O}}^f$ is the full subcategory of $R$ such that the square of the maximal ideal $\mathfrak{m}_R$ does not contain $p$.

\begin{Def}Let $f:R\rightarrow S$ be a surjective map in $\mathcal{C}_{\mathcal{O}}^f$. The map $f$ is said to be a \textit{small extension} if 
\begin{itemize}
\item the maximal ideal $\mathfrak{m}_R$ of $R$ is principal,
    \item the map $f$ identifies $S$ with $R/I$ where $I$ is an ideal in $R$ and $I\cdot \mathfrak{m}_R=0$.
\end{itemize}
\end{Def}

\begin{Def}\label{localdefconditions}
A functor of deformations \[\tilde{\mathcal{C}}_v:\mathcal{C}_{\mathcal{O}}\rightarrow \op{Sets}\] of $\bar{\rho}_{\restriction \op{G}_v}$ is called a \textit{deformation condition} if the following conditions $\eqref{dc1}$ to $\eqref{dc3}$ stated below are satisfied. If in addition, $\eqref{dc4}$ is also satisfied, then it is said to be a \textit{liftable deformation condition}:
      \begin{enumerate}
          \item\label{dc1} $\tilde{\mathcal{C}}_v(\F_q)=\bar{\rho}_{\restriction \G_v}.$
          \item\label{dc2} For $i=1,2$, let $R_i\in \mathcal{O}$ and $\rho_i\in \tilde{\mathcal{C}}_v(R_i)$. Let $I_1$ be an ideal in $R_1$ and $I_2$ an ideal in $R_2$ such that there is an isomorphism $\alpha:R_1/I_1\xrightarrow{\sim} R_2/I_2$ satisfying \[\alpha(\rho_1 \;\text{mod}\;{I_1})=\rho_2 \;\text{mod}\;{I_2}.\] Let $R_3$ be the fibred product \[R_3=\lbrace(r_1,r_2)\mid \alpha(r_1\;\text{mod}\; I_1)=r_2\; \text{mod} \;I_2\rbrace\] and $\rho_1\times_{\alpha} \rho_2$ the induced $R_3$-representation, then $\rho_1\times_{\alpha} \rho_2\in \tilde{\mathcal{C}}_v(R_3)$.
          \item\label{dc3} Let $R\in \mathcal{C}_{\mathcal{O}}$ with maximal ideal $\mathfrak{m}_R$. If $\rho\in \tilde{\mathcal{C}}_v(R)$ and $\rho\in \tilde{\mathcal{C}}_v(R/\mathfrak{m}_R^n)$ for all $n>0$ it follows that $\rho\in \tilde{\mathcal{C}}_v(R)$. In other words, the functor $\tilde{\mathcal{C}}_v$ is continuous.
          \item\label{dc4}
          For every small extension $R\rightarrow S$ the induced map $\tilde{\mathcal{C}}_v(R)\rightarrow \tilde{\mathcal{C}}_v(S)$ is surjective.
      \end{enumerate}
      \end{Def}
     Condition $\eqref{dc2}$ is referred to as the Mayer-Vietoris property. By a well-known result of Grothendieck \cite[section 18]{Mazurintro}, conditions $\eqref{dc1},\eqref{dc2}$ and $\eqref{dc3}$ guarantee that $\tilde{\mathcal{C}}_v$ is pro-represented by a scheme $\op{Spec} R_v^{\op{univ}}$, where $R_v^{\op{univ}}\in  \mathcal{C}_{\mathcal{O}}$. In other words, there is a deformation 
 \[\rho_v^{\op{univ}}:\G_v\rightarrow \GL_2(R_v^{\op{univ}})\]of $\bar{\rho}_{\restriction \op{G}_v}$ such that any deformation $\varrho\in \tilde{\mathcal{C}}_v(R)$ is induced by a unique map of coefficient rings $R_v^{\op{univ}}\rightarrow R$. In addition, condition $\eqref{dc4}$ implies that $\op{Spec} R_v^{\op{univ}}$ is formally smooth.
 Recall that $\Ad\bar{\rho}$ denotes the $\F_q$ vector space of $2\times 2$ matrices over $\F_q$ on which $\G_{\Q,S}$ acts via the adjoint action. Set $\Ad^0\bar{\rho}$ for the $\F_q[\G_{\Q,S}]$ submodule of trace zero matrices in $\Ad \bar{\rho}$. The functor of deformations of $\bar{\rho}_{\restriction \G_v}$ is denoted $\widetilde{\operatorname{Def}}_{v}$. Let $\kappa: \op{G}_{\Q,S}\rightarrow \GL_1(\mathcal{O})$ be a lift of $\det \bar{\rho}$. Denote by $\tilde{\mathcal{C}}_v^{\kappa}$ the subfunctor for which the determinant character is set to be $\kappa_v:=\kappa_{\restriction \op{G}_v}$. In greater detail, $\tilde{\mathcal{C}}_v^{\kappa}(R)$ consists of deformations $\varrho\in \tilde{\mathcal{C}}_v(R)$ such that $\det \varrho$ is equal to the composite \[\op{G}_v\xrightarrow{\kappa_v} \GL_1(\mathcal{O})\rightarrow \op{GL}_1(R),\] where the latter map is induced from the structure map $\mathcal{O}\rightarrow R$. In a discussion in which $\kappa$ is fixed, we set $\mathcal{C}_v:=\tilde{\mathcal{C}}_v^{\kappa}$ for ease of notation. Likewise, the subfunctor of deformations with fixed determinant is denoted by $\operatorname{Def}_v$. The \textit{infinitesimal deformations} $\widetilde{\operatorname{Def}}_{v}(\F_q[\epsilon]/(\epsilon^2))$ and $\operatorname{Def}_{v}(\F_q[\epsilon]/(\epsilon^2))$ are $\F_q$-vector spaces and are equipped with natural isomorphisms \[\widetilde{\op{Def}}_{v}(\F_q[\epsilon]/(\epsilon^2))\simeq  H^1(\G_v, \Ad\bar{\rho}) \text{, and } \op{Def}_{v}(\F_q[\epsilon]/(\epsilon^2))\simeq  H^1(\G_v, \Ad^0\bar{\rho}).\] The infinitesimal deformation associated with the cohomology class $f\in H^1(\op{G}_v,  \Ad\bar{\rho})$ is \[(\op{Id}+f \epsilon) \bar{\rho}:\op{G}_{v}\rightarrow \op{GL}_2(\F_q[\epsilon]/(\epsilon^2)).\]
 Set $\tilde{\mathcal{N}}_v:=\tilde{\mathcal{C}}_v(\F_q[\epsilon])\subseteq H^1(\G_v, \Ad\bar{\rho})$, and $\mathcal{N}_v:=\mathcal{C}_v(\F_q[\epsilon])\subseteq H^1(\G_v, \Ad^0\bar{\rho})$. The space $\tilde{\mathcal{N}}_v$ can be identified with the tangent space of $R_v^{\op{univ}}$. The following standard fact is noted in \cite[Fact 5]{KhareRamakrishna} and proceeds from the discussion on local deformation conditions in \cite{RamakrishnaFM}.
\begin{Fact}\label{Fact1}
Fix a lift $\kappa:\op{G}_{\Q,S}\rightarrow \GL_1(\mathcal{O})$ of $\det \bar{\rho}$ and for each prime $v$, set $\kappa_v:=\kappa_{\restriction \op{G}_v}$. For $v\in S\backslash\{p\}$, there exists a liftable local deformation condition $\mathcal{C}_v$ of $\bar{\rho}_{\restriction \op{G}_v}$ with fixed determinant $\kappa_v$ such that $\dim \mathcal{N}_v=h^0(\op{G}_v, \Ad^0 \bar{\rho})$. These local deformation conditions are described in the proof of \cite[Proposition 1]{RamakrishnaFM} on a case by case basis. At the prime $p$, let $\mathcal{C}_p$ consist of ordinary deformations of $\bar{\rho}_{\restriction \op{G}_p}$ fixed determinant $\kappa_v$. The functor $\mathcal{C}_p$ is also a liftable deformation condition and $\dim \mathcal{N}_p=h^0(\op{G}_p, \Ad^0 \bar{\rho})+1$. Set $\tilde{\mathcal{C}}_p$ denote the deformation functor of ordinary deformations of $\bar{\rho}_{\restriction \op{G}_p}$ (for which the determinant is not fixed). For $v\neq p$, let $\tilde{\mathcal{C}}_v$ consist of the unramified central twists of $\mathcal{C}_v$.
\end{Fact}
Let $R\rightarrow R/I$ be a small extension and $t$ a generator of the maximal ideal of $R$. Let $v\in S$ and $\tilde{\mathcal{C}}_v$ as above. Let $\varrho\in \tilde{\mathcal{C}}_v(R)$ and $\varrho_0:=\varrho\mod{I}$. The twist of $\varrho$ by a cohomology class $X\in \tilde{\mathcal{N}}_v$ is precribed by
$\operatorname{exp}(X\otimes t)\varrho:=(\operatorname{Id}+X t)\varrho$. Furthermore, the twist is a deformation of $\varrho_0$. The fibers of $\varrho_0$ w.r.t the mod $I$ reduction map $\tilde{\mathcal{C}}_v(R)\rightarrow \tilde{\mathcal{C}}_v(R/I)$ is an $\tilde{\mathcal{N}}_v$-pseudotorsor. Likewise, the fibers of $\mathcal{C}_v(R)\rightarrow \mathcal{C}_v(R/I)$ is an $\mathcal{N}_v$-pseudotorsor.\begin{Fact}\label{tildeNdef}
\begin{enumerate}
\item For $v\neq p$, we have that $\tilde{\mathcal{N}}_v=\mathcal{N}_v\oplus H_{\operatorname{nr}}^1(\G_v, \F_q)$ and the dimension of $\tilde{\mathcal{N}}_v$ is equal to $h^0(\op{G}_v, \op{Ad}\bar{\rho})$. It is easy to show that the dimension of $\tilde{\mathcal{N}}_v$ is equal to $h^0(\op{G}_v, \op{Ad}\bar{\rho})$ for $v\in S\backslash \{p\}$.
\item For $v=p$, the tangent space $\tilde{\mathcal{N}}_p$ properly contains the direct sum of $\mathcal{N}_p\oplus H^1_{\operatorname{nr}}(\G_p, \F_q)$. Set $W=\mtx{*}{*}{}{}$, observe that $W$ is stable under conjugation by upper triangular matrices. We have that \[\tilde{\mathcal{N}}_p=\ker\{H^1(\G_p, \Ad\bar{\rho})\rightarrow H^1(I_p, \Ad\bar{\rho}/W)\}\] as the choice of tangent space at $p$.
\end{enumerate}
\end{Fact}
\subsection{Dimension Calculations}
\par  In this section, we show that the dimension of $\tilde{\mathcal{N}}_p$ is equal to $h^0(\op{G}_p, \op{Ad}\bar{\rho})+2$. Let $U\subset \Ad^0\bar{\rho}$ be the subspace of upper triangular matrices $\mtx{a}{b}{0}{-a}$, and $U^0\subset U$ be the subspace of strictly upper triangular matrices $\mtx{0}{b}{0}{0}$. Since $\bar{\rho}=\mtx{\varphi}{\ast}{0}{1}$, both $U^0$ and $U$ are Galois stable.
\par 
Let $\tilde{U}$ denote the Galois stable space of upper triangular matrices $\mtx{a}{b}{0}{d}$, we have that $\tilde{U}=U\oplus \F_q \operatorname{Id}$. The quotient $\Ad\bar{\rho}/{\tilde{U}}=\F_q(\varphi^{-1})$ has no fixed points for the action of $\G_p$ and hence $H^1_{\operatorname{nr}}(\G_p, \Ad\bar{\rho}/{\tilde{U}})=0$.
Consequently, $\tilde{\mathcal{N}}_p$ may be identified with the subspace \[\tilde{\mathcal{N}}_p=\ker\left\lbrace H^1(\G_p, \tilde{U})\rightarrow H^1(I_p, \tilde{U}/W)\right\rbrace\] of $H^1(\G_p, \tilde{U})$.
\par For the decomposition
\[H^1(\G_p,\tilde{U})\xrightarrow{\sim} H^1(\G_p,U)\oplus H^1(\G_p, \F_q\cdot \operatorname{Id})\]
let $\pi_1$ and $\pi_2$ denote the projection maps to
$H^1(\G_p,U)$ and $H^1(\G_p,\F_q \cdot \operatorname{Id})$ respectively.

The map on restriction $\pi_1'={\pi_1}_{\restriction \tilde{\mathcal{N}}_p}$ induces an exact sequence.
\begin{Prop}
The map on restriction $\pi_1'={\pi_1}_{\restriction \tilde{\mathcal{N}}_p}$ induces a short exact sequence \[0\rightarrow H^1_{\operatorname{nr}}(\G_p, \F_q \cdot \operatorname{Id})\rightarrow \tilde{\mathcal{N}}_p\xrightarrow{\pi_1'} H^1(\G_p,U)\rightarrow 0.\]
\end{Prop}
\begin{proof}
As noted earlier, $\tilde{\mathcal{N}}_p$ is identified with the subspace \[\tilde{\mathcal{N}}_p=\ker\left\lbrace H^1(\G_p, \tilde{U})\rightarrow H^1(I_p, \tilde{U}/W)\right\rbrace\] of $H^1(\G_p, \tilde{U})$. The kernel of the projection $\pi_1:H^1(\G_p,\tilde{U})\rightarrow H^1(\G_p, U)$ is $H^1(\G_p, \F_q\cdot \operatorname{Id})$. As a result, the kernel of $\pi_1'$ is the intersection 
$H^1(\G_p, \F_q \cdot \operatorname{Id})\cap \tilde{\mathcal{N}}_p$, which consists of unramified central classes $H^1(\G_p,\F_q\cdot \operatorname{Id})$. This shows that the sequence is exact in the middle.
\par The map on the left is clearly injective. We prove that the map $\pi_1'$ is surjective.  Let $f\in H^1(\G_p, U)$, we show that $f$ is in the image of $\pi_1'$. It suffices to show that there exists an element $f'\in H^1(\G_p, \F_q\cdot \operatorname{Id})$ such that $f-f'\in \tilde{\mathcal{N}}_p$, i.e. the class $f-f'$ must map to zero in $H^1(I_p, \tilde{U}/W)$.
\par We observe that the composite of the inclusion of $\F_q\cdot \operatorname{Id}$ into $\tilde{U}$ with the quotient $\tilde{U}\rightarrow \tilde{U}/W$ is an isomorphism of $\G_p$ modules (both spaces are fixed by $\G_p$). We simply take $f_1$ to coincide with $f$ modulo $W$ w.r.t this isomorphism. This completes the proof.
\end{proof}
\begin{Prop}
For the dimension of $H^1(\G_p, U)$ we are to consider two cases
\[h^1(\G_p, U)=\begin{cases}2\text{ if }\bar{\rho}_{\restriction \G_p}\text{ is indecomposable}\\
3\text{ otherwise if }\bar{\rho}_{\restriction \G_p}\text{ is a sum of characters}
\end{cases}\] in other words,
\[h^1(\G_p, U)=2+h^0(\G_p, \Ad^0\bar{\rho}).\]
Consequently, dimension of $\tilde{\mathcal{N}}_p$ is
\[\begin{split}\dim \tilde{\mathcal{N}}_p= & h^1_{\operatorname{nr}}(\G_p,\F_q\cdot \operatorname{Id})+ h^1(\G_p, U)\\
=& 3+h^0(\G_p, \Ad^0\bar{\rho})\\
=& 2+h^0(\G_p, \Ad\bar{\rho}).
\end{split}
\]
\end{Prop}
\begin{proof}
 We will make use of the cohomology sequence associated with the short exact sequence
\[0\rightarrow U^0\rightarrow U\rightarrow U/U^0\rightarrow 0.\] From the isomorphism $U^0\simeq \F_q(\varphi)$ and the fact that $\varphi_{\restriction \G_p}\neq 1, \bar{\chi}^{-1}_{\restriction \G_p}$. The assumptions on $\varphi_{\restriction \G_p}$ ensure that $H^2(\G_p, U^0)=0$. From the Euler characteristic formula we have that $ H^1(\G_p, U^0)$ is $1$ dimensional.
\par From the Euler characteristic formula, 
\[h^1(\G_p, U/U^0)=2\]
and 
\[h^1(\G_p, U^0)=1.\] From the long exact sequence, we have the following formula for the dimension of $H^1(\G_p,U)$
\[\begin{split}h^1(\G_p,U)= &h^1(\G_p, U^0)+h^1(\G_p, U/U^0)+h^0(\G_p, U)-1\\
=& 2+h^0(\G_p,U)\\
=& \begin{cases} 2 \text{ if } \bar{\rho}_{\restriction \G_p} \text{ is indecomposable}\\
3 \text{ otherwise.}
\end{cases}
\end{split}\]
\end{proof}
\par Condition \ref{c8ofmain} requires that $\varphi_{\restriction I_p}=\chi^{k-1}_{\restriction I_p}$ where $2\leq k\leq p-1$.
We first enumerate the possibilities which arise when $k\geq 3$.
\begin{Prop}\label{KhareRamProp}
Suppose the weight $k\neq 1,2$ so that $3\leq k\leq p-1$. Recall that if $k=p-1$ it is stipulated that $\varphi_{\restriction \G_p}\neq \bar{\chi}^{-1}_{\restriction \G_p}$. Let $\tau= \varphi_{\restriction \G_p}$ be a product of an unramified character with $\bar{\chi}^{k-1}$. With these assumptions, there are the following cases to consider
\begin{enumerate}
\item 
$\bar{\rho}_{\restriction \G_p}=\mtx{\tau}{}{}{1}$ up to twisting by an unramified character. The ordinary arbitrary-weight deformation ring is smooth in $4$ variables and $\dim \tilde{\mathcal{N}}_p=4$, $h^0(\G_p, \Ad\bar{\rho})=2$,
or,
\item
$\bar{\rho}_{\restriction \G_p}=\mtx{\tau}{*}{}{1}$ is indecomposable in which case, the ordinary arbitrary-weight deformation ring is smooth in $3$ variables and $\dim \tilde{\mathcal{N}}_p=3$, $h^0(\G_p, \Ad\bar{\rho})=1$.
\end{enumerate}
\end{Prop}
\begin{proof}
\par 
There is an unramified character $\eta:\G_p\rightarrow \F_q^{\times}$ such that $\tau=\eta \bar{\chi}^{k-1}_{\restriction \G_p}$. It follows from the assumptions on $k$ and $\varphi$ that the characters $\tau,\tau \bar{\chi}_{\restriction \G_p}$ and $\tau \bar{\chi}_{\restriction \G_p}^{-1}$ are all nontrivial. In both cases we compare the ordinary deformation problem to the upper triangular deformation problem. Observe that $\Ad\bar{\rho}/\tilde{U}$ can be identified with $\F_q(\tau^{-1})$.
\par Since $\tau\neq 1$, we have that $H^1_{\operatorname{nr}}(\G_p, \Ad\bar{\rho}/\tilde{U})=0$
and consequently 
\[\tilde{\mathcal{N}}_p=\ker\{H^1(\G_p, \tilde{U})\rightarrow H^1(I_p, \tilde{U}/W)\}\] where we recall that $W\subset \tilde{U}$ consists of matrices $\mtx{\ast}{\ast}{0}{0}$. 
\par In both cases $\tilde{\mathcal{C}}_p$ can be derived from the upper triangular deformation problem by imposing the condition that a generator of $Gal(\Q^{cyc}/\Q)\simeq \mathbb{Z}_p$ maps to an upper triangular matrix for which the lower right entry is trivial.
\par
1. In the first case, $\tilde{U}\simeq \F_{p}^2\oplus \F_q(\tau)$. An application of Local-Duality implies that 
\[\begin{split}
& h^2(\G_p, \tilde{U})\\=&h^0(\G_p, \tilde{U}^*)\\
=&2h^0(\G_p, \F_q(\bar{\chi}_{\restriction \G_p}))+h^0(\G_p,\F_q(\tau^{-1}\bar{\chi}_{\restriction \G_p}))\\
=&0.
\end{split}\]
An application of the Euler Characteristic Formula shows that
\[\begin{split}
&h^1(\G_p,\tilde{U})\\=&h^0(\G_p,\tilde{U})+h^2(\G_p,\tilde{U})+\dim \tilde{U}\\
=&2+0+3=5.
\end{split}\]
Consequently, the deformation ring of unrestricted upper triangular deformations of $\bar{\rho}_{\restriction \G_p}$ is a power series ring in five variables. The relation imposed for a deformation to be ordinary comes from setting the ramified part of the lower right entry when evaluated at a topological generator of the cyclotomic extension of $\Q_p$. This restriction cuts down dimension of the tangent space by one and corresponds to going modulo one of the variables in the tangent space. We can take this to be a variable in a power series ring in five variables. The deformation ring is a power series ring in four variables, for further details the reader is referred to the proof of part (1) of \cite[Proposition 11]{KhareRamakrishna}. This completes the proof of the first part.
\par Next, we prove part (2). Since the diagonal characters of $\tilde{U}^*$ are nontrivial, it follows that $H^2(\G_p,\tilde{U})=0$. Since we are assuming that $\bar{\rho}_{\restriction \G_p}$ is indecomposable $h^0(\G_p, \tilde{U})=1$. An application of the Euler Characteristic Formula shows that
\[\begin{split}
&h^1(\G_p,\tilde{U})\\=&h^0(\G_p,\tilde{U})+h^2(\G_p,\tilde{U})+\dim \tilde{U}\\
=&1+0+3=4.
\end{split}\] The rest of the proof is identical to that of part (2) of \cite[Proposition 11]{KhareRamakrishna}.

\end{proof}

\par Proposition \ref{KhareRamProp} generalizes the first two cases of \cite[Proposition 11]{KhareRamakrishna}, in which $k$ is prescribed to be $2$. We refer to \cite[Proposition 11]{KhareRamakrishna} for an enumeration of all the possibilities for $k=2$.
\subsection{Selmer Groups}
\par Let $M$ be an $\F_q[\G_{\Q,S}]$-module, and $Z$ a finite set of places containing $S$. A Selmer condition $\mathcal{L}$ over the set of primes $Z$ is a collection of subspaces $\mathcal{L}_v\subset H^1(\op{G}_v, M)$ as $v$ ranges over $Z$. Let $\mathcal{L}_v^{\perp}\subset H^1(\op{G}_v, M^*)$ denote the annihilator of $\mathcal{L}_v$ w.r.t nondegenerate local pairing,
\[H^1(\op{G}_v, M)\times H^1(\op{G}_v, M^*)\rightarrow H^2(\op{G}_v, \F_q(1))\simeq \F_q.\] The Selmer group associated to $\mathcal{L}$ is 
\[H^1_{\mathcal{L}}(\G_{\Q,Z}, M):=\ker\left\lbrace H^1(\G_{\Q, Z}, M)\rightarrow \bigoplus_{v\in Z} \frac{H^1(\op{G}_v, M)}{\mathcal{L}_v}\right\rbrace\]
and the dual Selmer group is the Selmer group associated to the dual Selmer condition $\mathcal{L}^{\perp}$
\[H^1_{\mathcal{L}^{\perp}}(\G_{\Q,Z}, M^*):=\ker\left\lbrace H^1(\G_{\Q, Z}, M^*)\rightarrow \bigoplus_{v\in Z} \frac{H^1(\op{G}_v, M^*)}{\mathcal{L}^{\perp}_v}\right\rbrace.\]
 The Proposition below is due to Wiles and follows from the Poitou-Tate long exact sequence for finite modules $M$.
\begin{Th}\cite[Theorem 8.7.9]{NSW}\label{wilesprop}
Suppose $M$ is an $\F_q[\G_{\Q,S}]$ module with finite cardinality and $Z$ a finite set of primes which contains $S$. Let $\mathcal{L}$ be a Selmer condition on $Z\cup \{\infty\}$. The dimensions of the Selmer and dual Selmer groups are related as follows
\[\begin{split}& h^1_{\mathcal{L}}(\G_{\Q,Z}, M)-h^1_{\mathcal{L}^{\perp}}(\G_{\Q,Z}, M^*)\\
=&h^0(\G_{\Q},M)-h^0(\G_{\Q}, M^*)+\sum_{v\in Z\cup \{\infty\}}(\dim \mathcal{L}_v-h^0(\op{G}_v, M)).
\end{split}\]
\end{Th}
 In the above formula, $\op{G}_{\infty}=\langle c\rangle$. The spaces $\mathcal{N}_{\infty}$ and $\tilde{\mathcal{N}}_{\infty}$ are set to be $0$. Recall that the dimension of $\tilde{\mathcal{N}}_v$ is $h^0(\op{G}_v, \Ad\bar{\rho})$ for $v\in S\backslash \{p\}$ and that by Proposition $\ref{KhareRamProp}$, the dimension of $\tilde{\mathcal{N}}_p$ is equal to $h^0(\op{G}_p, \Ad \bar{\rho})+2$. Since $\bar{\rho}$ is odd, it follows that $H^0(\op{G}_{\infty}, \Ad \bar{\rho})$ consists of diagonal matrices and hence is two dimensional. Observe that $h^0(\G_{\Q},\Ad \bar{\rho})=1$, $h^0(\G_{\Q}, \Ad \bar{\rho}^*)=0$ and
 \[\sum_{v\in S\cup \{\infty\}} (\dim \tilde{\mathcal{N}}_v-h^0(\op{G}_v, \Ad \bar{\rho}))=0.\]Therefore it is a consequence of Theorem $\ref{wilesprop}$ that
\begin{equation}\label{selmerdselmerdiff1}h^1_{\tilde{\mathcal{N}}}(\op{G}_{\Q,S}, \Ad \bar{\rho})-h^1_{\tilde{\mathcal{N}}^{\perp}}(\op{G}_{\Q,S}, \Ad \bar{\rho}^*)=1.\end{equation}

\section{Versal Deformations at Trivial Primes}\label{highlyversal}
\par In this section, we recall the deformation problems at trivial primes due to Hamblen-Ramakrishna. Throughout, we fix a lift $\kappa:\op{G}_{\Q,S}\rightarrow \GL_1(\mathcal{O})$ and study two types of deformation problems $\tilde{\mathcal{C}}_v : \mathcal{C}_{\mathcal{O}}\rightarrow \op{Sets}$, and $\mathcal{C}_v : \mathcal{C}_{\mathcal{O}}\rightarrow \op{Sets}$ for which the determinant is fixed to be $\kappa_v$ for the latter condition.
\begin{Def}
A prime $v$ is a trivial prime if
\begin{enumerate}
\item
$\G_v\subset \ker \bar{\rho}$,
\item
the prime $v\equiv 1\mod{p}$ but $v\not\equiv 1\mod{p^2}$.
\end{enumerate}

\end{Def}
The following fact is an easy consequence of local-duality and the local Euler-characteristic formula (cf. \cite{NSW}).
\begin{Fact}
For a trivial prime $v$, we have the following dimension calculations:\[
h^i(\op{G}_v, \Ad^0\bar{\rho})=\begin{cases}3\text{ if }i=0,2,\\
6\text{ if }i=1,\\
\end{cases} \text{and }\hspace{0.25 cm}
h^i(\op{G}_v, \Ad\bar{\rho})=\begin{cases}4\text{ if }i=0,2,\\
8\text{ if }i=1.\\
\end{cases}\]
All higher cohomology groups are zero.
\end{Fact}
\par Let $v$ be a trivial prime. Choose a square root of $\kappa(\sigma_v)v^{-1}$ in $\mathcal{O}$. Fix a basis with respect to which $\bar{\rho}$ is upper triangular. Let $v$ be a trivial prime, $\bar{\rho}_{\restriction G_v}$ is tamely ramified. Let $\sigma_v$ be a choice of Frobenius at $v$ and $\tau_v$ a generator of the maximal pro-$p$ tame inertia. The maximal pro-$p$ extension of $\Q_v$ is generated by $\sigma_v$ and $\tau_v$ which are subject to a single relation $\sigma_v \tau_v \sigma_v^{-1}=\tau_v^v$.
\begin{Def}\label{trivialdeformationconditions} We proceed to prescribe deformation problems at trivial primes.
\begin{itemize}
\item
Let $\mathcal{D}_v$ be the class of deformations of $\bar{\rho}_{\restriction G_v}$ containing representatives taking 
\[
\begin{split}
&\sigma_v\mapsto (\kappa(\sigma_v)v^{-1})^{1/2}\mtx{v}{x}{0}{1}\\
&\tau_v\mapsto \mtx{1}{y}{0}{1}
\end{split}\]
with respect to a basis lifting $\mathscr{B}$ and insist that $p^2$ divides $x$.
\item Separate $\mathcal{D}_v$ into two classes, the first class of deformations are those that are ramified modulo $p^2$, i.e. $\mathcal{D}_v^{\operatorname{ram}}$ consists of those deformations for which $p$ divides $y$ but $p^2$ does not divide $y$. Those that are unramified modulo $p^2$ are denoted $\mathcal{D}_v^{\operatorname{nr}}$. \item We let $\tilde{\mathcal{D}}_v$ be the deformations obtained as unramified central twists of those in $\mathcal{D}_v$. In other words, $\tilde{\mathcal{D}}_v$ are those deformations with representative taking 
\[\sigma_v\mapsto z\mtx{v}{x}{0}{1}\\
\text{ and }\tau_v\mapsto \mtx{1}{y}{0}{1}\]
where $z\equiv 1 \mod{p}$, $x\equiv 0 \mod{p^2}$. We emphasize that the requirement that $x\equiv 0\mod{p^2}$ is an additional assumption in our setting.
\end{itemize}
\end{Def}
 The space $\mathcal{N}_v$ has dimension equal to $h^0(\G_v, \Ad^0\bar{\rho})=3$. This facilitates for a balanced Selmer condition, i.e. \[h^1_{\mathcal{N}}(\op{G}_{\Q,S\cup X}, \Ad^0\bar{\rho})=h^1_{\mathcal{N}^{\perp}}(\op{G}_{\Q,S\cup X}, \Ad^0\bar{\rho}^*).\]
\begin{Cor}\label{SelmerdualSelmer1}
Let $\bar{\rho}$ be subject to the conditions enumerated in Theorem $\ref{main}$. Let $X$ be a finite set of trivial primes disjoint from $S$. Then
\[h^1_{\tilde{\mathcal{N}}}(\op{G}_{\Q,S\cup X}, \Ad \bar{\rho})-h^1_{\tilde{\mathcal{N}}^{\perp}}(\op{G}_{\Q,S\cup X}, \Ad \bar{\rho}^*)=1.\]
\end{Cor}
\begin{proof}
The dimension of the local tangent spaces at $p$ for the full-adjoint deformation problems are as follows, 
\[
\begin{split}
& \dim \tilde{\mathcal{N}}_{\infty}=0\\
&\dim \tilde{\mathcal{N}}_p=h^0(\G_p, \Ad\bar{\rho})+2\\
&\dim \tilde{\mathcal{N}}_v=h^0(\op{G}_v, \Ad\bar{\rho}) \text{ for $v\in S\backslash \{p\}$}\\
&\dim \tilde{\mathcal{N}}_v=h^0(\op{G}_v, \Ad\bar{\rho})\text{ for $v\notin S$ a trivial prime.}
\end{split}\] An application of Wiles' formula yields
\begin{equation}\label{SelmerDSelmer}
\begin{split}
&h^1_{\tilde{\mathcal{N}}}(\op{G}_{\Q,S\cup X}, \Ad \bar{\rho})-h^1_{\tilde{\mathcal{N}}^{\perp}}(\op{G}_{\Q,S\cup X}, \Ad \bar{\rho}^*)\\
&=h^0(\G_{\Q}, \Ad\bar{\rho})-h^0(\G_{\Q}, \Ad\bar{\rho}^*)+\sum_{v\in S\cup X\cup\{\infty\}}\left(\dim \tilde{\mathcal{N}}_v-h^0(\op{G}_v, \Ad\bar{\rho})\right)\\
&=1+\sum_{v\in S\cup X\cup\{\infty\}}\left(\dim \tilde{\mathcal{N}}_v-h^0(\op{G}_v, \Ad\bar{\rho})\right)\\
&=3-h^0(\op{G}_{\infty}, \Ad\bar{\rho})\\
&=1.
\end{split}
\end{equation}

\end{proof}
\par The key observation of Hamblen and Ramakrishna is that one may allow ramification at a number of trivial primes $X_1$ disjoint from $S$ so as to lift $\bar{\rho}$ to an irreducible mod $p^3$ representation $\rho_3$ satisfying the local constraints at the set $S$ and the set of trivial primes $X_1$. At this stage, the versal deformation functors at trivial primes play the exact role of the deformation conditions at nice primes. It becomes possible to adapt Ramakrishna's lifting argument from \cite{RamakrishnaFM}, namely, allow ramification at a finite set of trivial primes $X\supseteq X_1$ so the dual Selmer group $H^1_{\mathcal{N}^{\perp}}(\G_{\Q,S\cup X}, \Ad^0\bar{\rho}^*)=0$. The Galois representation $\rho_3$ lifts to a characteristic zero representation $\rho:\G_{\Q,S\cup X}\rightarrow \GL_2(\mathcal{O})$ which is irreducible and odd. It follows from the results of Skinner and Wiles \cite{skinnerwiles} that $\rho$ arises from a Hecke eigencuspform. The reader may also refer to the treatment in \cite{ray3}, which discusses Ramakrishna's lifting construction in the residually reducible case.
\par The two classes of deformations (ramified mod $p^2$ and unramified mod $p^2$) play different roles in the deformation theoretic arguments. Trivial primes at which deformations are ramified modulo $p^2$ are used to lift $\bar{\rho}$ to a mod $p^3$ representation $\rho_3$ satisfying local constraints. The trivial primes at which deformations are unramified modulo $p^2$ are used to kill the dual Selmer group and thereby lift $\rho_3$ to a characteristic zero rerpesentation. The reader is not required to fully understand the role of the two choices of deformation problems to understand and appreciate the arguments in this manuscript.

We recall some notation from \cite[section 4]{hamblenramakrishna}.
\begin{Def}
Set
\[
\begin{split}
& f_1(\sigma_v)=\mtx{0}{1}{0}{0} \text{, }f_1(\tau_v)=\mtx{0}{0}{0}{0},\\ &
f_2(\sigma_v)=\mtx{0}{0}{0}{0} \text{, }f_2(\tau_v)=\mtx{0}{1}{0}{0},\\ & g^{\operatorname{nr}}(\sigma_v)=\mtx{0}{0}{1}{0} \text{, }g^{\operatorname{nr}}(\tau_v)=\mtx{0}{0}{0}{0},\\ & g^{\operatorname{ram}}(\sigma_v)=\mtx{0}{0}{1}{0} \text{, }g^{\operatorname{ram}}(\tau_v)=\mtx{-\frac{y}{v-1}}{0}{0}{\frac{y}{v-1}}.
\end{split}\]
The space $\mathcal{Q}_v$ denotes the subspace of $H^1(G_v, \Ad^0\bar{\rho})$ spanned by $f_1$ and $f_2$, let $\mathcal{P}_v^{\operatorname{nr}}$ the subspace spanned by $f_1,f_2$ and $g^{\operatorname{nr}}$ and $\mathcal{P}_v^{\operatorname{ram}}$ the subspace spanned by $f_1,f_2$ and $g^{\operatorname{ram}}$.
\end{Def}
\begin{Def} 
\begin{enumerate}
\item
For trivial primes $v$ whose mod $p^2$ deformations are unramified, set $\mathcal{M}_v$, $\mathcal{N}_v$ and $\mathcal{C}_v$ to consist of conjugates by $\mtx{1}{0}{1}{1}$ of elements of $\mathcal{Q}_v$, $\mathcal{P}_v^{\operatorname{nr}}$ and $\mathcal{D}_v^{\operatorname{nr}}$ respectively.
\item
For trivial primes $v$ whose mod $p^2$ deformations are ramified, set $\mathcal{M}_v$, $\mathcal{N}_v$ and $\mathcal{C}_v$ consist of conjugates by $\mtx{0}{1}{1}{0}$ of elements of $\mathcal{Q}_v$, $\mathcal{P}_v^{\operatorname{ram}}$ and $\mathcal{D}_v^{\operatorname{ram}}$ respectively.
\item In both cases, let $\tilde{\mathcal{N}}_v$, $\tilde{\mathcal{M}}_v$ and $\tilde{\mathcal{C}}_v$ be the unramified central twists of ${\mathcal{N}}_v$, ${\mathcal{M}}_v$ and ${\mathcal{C}}_v$ respectively.
\end{enumerate}
\end{Def}

\begin{Def}\label{decreasingideals}
A ring $R\in \Cfh$ with maximal ideal $\mathfrak{m}_R$ is endowed with a decreasing chain of ideals $\{\mathfrak{n}_k\}_{k\geq 1}$ defined by
$\mathfrak{n}_k:=pR\cap\mathfrak{m}_R^k$. In dealing with more than one ring $R$ we will use $\mathfrak{n}_k(R)$ instead of $\mathfrak{n}_k$.
\end{Def}
\begin{Prop}\label{furtherdetailed}
Let $R\in \mathfrak{C}$ a coefficient ring with maximal ideal $\mathfrak{m}$ ( recall that it is stipulated that $p\notin \mathfrak{m}^2$). For both ramified and unramified mod $p^2$ deformations, the space $\mathcal{N}_v\otimes \mathfrak{n}_k/\mathfrak{n}_{k+1}$ preserves $\mathcal{C}_v(R/\mathfrak{n}_{k+1})$ for $k\geq 2$.
\end{Prop}
\begin{proof}
\par Let $k\geq 2$ and $\varrho_k\in \mathcal{C}_v(R/\mathfrak{n}_{k+1})$ and let $pr\in \mathfrak{n}_k$. Consider the case when $p^2$ divides $y$. Since $k\geq 2$ the assumption $p\notin \mathfrak{m}^2$ implies that $p\notin \mathfrak{n}_k$ and consequently, $r\in \mathfrak{m}$. Let \[\varrho_k':=\exp(g^{\operatorname{nr}}\otimes pr)\varrho_k=(\op{Id}+prg^{\operatorname{nr}})\varrho_k\] we show that $\varrho_k'$ belongs to $\mathcal{C}_v(R/\mathfrak{n}_{k+1})$.
\par We recall that \[\varrho_k(\sigma_v)=(\kappa(\sigma_v) v^{-1})^{1/2}\mtx{v}{x}{0}{1}.\] It is easy to see that $p\mathfrak{n}_k\subseteq \mathfrak{n}_{k+1}$. Since $v$ is $1$ mod $p$ and $x\in \mathfrak{m}$, 
\[
\begin{split}&(\op{Id} + prg^{\operatorname{nr}})\varrho_k(\sigma_v)=(\kappa(\sigma_v) v^{-1})^{1/2}\mtx{v}{x}{pr}{1} \\
&(\op{Id} + pr g^{\operatorname{nr}})\varrho_k(\tau_v)=\mtx{1}{y}{0}{1}.\end{split}
\]
Since $v\not\equiv 1 \mod{p^2}$, we have that $\frac{p}{v-1}$ is a unit. Since the element $x$ is divisible by $p^2$ we have that $\frac{xpr}{v-1}\in \mathfrak{n}_{k+1}$. We see that
\begin{equation}\label{unchanged}\left(\op{Id} +r\mtx{0}{0}{\frac{p}{v-1}}{0}\right)^{-1}(\kappa(\sigma_v) v^{-1})^{1/2}\mtx{v}{x}{pr}{1}\left(\op{Id} +r\mtx{0}{0}{\frac{p}{v-1}}{0}\right)\end{equation}
\[=(\kappa(\sigma_v) v^{-1})^{1/2}\mtx{v}{x}{0}{1}=\varrho_k(\sigma)\]
and as $p^2$ divides $y$, the element $\frac{pry}{v-1}\in \mathfrak{n}_{k+1}$ and as a consequence,
\[\left(\op{Id} +r\mtx{0}{0}{\frac{p}{v-1}}{0}\right)^{-1}\mtx{1}{y}{0}{1}\left(\op{Id} +r\mtx{0}{0}{\frac{p}{v-1}}{0}\right)\]
\[=\mtx{1}{y}{0}{1}=\varrho_k(\tau).\]
\par 
Next consider deformations for which $p^2$ does not divide $y$. The conclusion of $\ref{unchanged}$ remains unchanged as $g^{\operatorname{ram}}(\sigma_v)=g^{\operatorname{nr}}(\sigma_v)$. We observe that
\[\left(\op{Id} +r\mtx{0}{0}{\frac{p}{v-1}}{0}\right)\varrho_k(\tau)\left(\op{Id} +r\mtx{0}{0}{\frac{p}{v-1}}{0}\right)^{-1}\]\[=(\op{Id} + pr g^{\operatorname{ram}})\varrho_k(\tau)=\mtx{1-\frac{pr y}{v-1}}{y}{0}{1+\frac{pr y}{v-1}}\] (unlike in the case for which $p^2$ divides $y$ this matrix need not be unipotent).
The case for $f_1$ and $f_2$ follows similarly.
\end{proof}

\section{A Purely Galois Theoretic Lifting Construction}\label{section4}
\par In \cite[section 5]{hamblenramakrishna}, it is shown that there exists a finite set of trivial primes $X_1$, that are disjoint from $S$ such that $\bar{\rho}$ lifts to an irreducible mod $p^3$ representation 
\[\rho_3:\G_{\Q,S\cup X_1}\rightarrow \GL_2(\mathcal{O}/p^3)\] which satisfies the liftable local deformation problems $\mathcal{C}_v$ at each prime $v\in S\cup X_1$. The set of trivial primes $X_1$ in particular has the property that $\Sh^2_{S\cup X_1}(\Ad^0\bar{\rho})=0$. It is then shown that the set of primes $X_1$ may be further extended to a finite set of trivial primes $X$ containing $X_1$ such that $h^1_{\mathcal{N}^{\perp}}(\op{G}_{\Q,S\cup X}, \Ad^0 \bar{\rho}^*)=0$. Denote by $\Phi$ the collection of versal deformation problems 
$\Phi=\{\Cv\}_{v\in S\cup X}$. Denote by $\rho_2:=\rho_3\mod{p^2}$.

\begin{Def}\label{DefDef}
For $R\in \mathfrak{C}$, let $\rho_{2,R}$ denote the deformation with image in $\GL_2(R/\mathfrak{n}_2)$ induced from $\rho_2$ by the structure map $\mathcal{O}\rightarrow R$. Let $\Df:\mathfrak{C}\rightarrow \operatorname{Sets}$ be the functor such that $\Df(R)$ consists of deformations $\rho_R:\G_{\Q,S\cup X}\rightarrow \GL_2(R)
$ for which
\[\rho_{2,R}= \rho_R\mod{\mathfrak{n}_2}.\]
\end{Def}
\begin{Def}
Let $R\in \Cfh$ have maximal ideal $\mathfrak{m}_R$ and let $J$ an ideal in $R$. The map of coefficient rings $R\rightarrow R/J$ is said to be \textit{nearly small} $\mathfrak{m}_RJ=0$. Recall that it is said to be \textit{small} if it is nearly small and if $J$ is principal.
\end{Def}
Let $R\rightarrow R/J$ be nearly small and $\mathfrak{m}_R$ be the maximal ideal of $R$. Since there is a unique isomorphism $R/\mathfrak{m}_R\simeq \F_q$ of $\mathcal{O}$-algebras, we simply identify the residue field $R/\mathfrak{m}_R$ with $\F_q$. Recall that by definition, $\mathfrak{m}_RJ=0$. For $a\in \F_q$ and $j\in J$, let $aj$ denote $\tilde{a}j$ where $\tilde{a}\in R$ is a lift of $a$. Note that $aj$ is well defined since $\tilde{a}j$ is independent of the choice of lift $\tilde{a}$. Therefore $J$ is a vector space over $\F_q$. Associate to a pure tensor 
\[\mtx{a}{b}{c}{d}\otimes j\in \Ad \bar\rho\otimes_{\F_q} J,\] the matrix 
\[\op{Id}+\mtx{aj}{bj}{cj}{dj}=\mtx{1+aj}{bj}{cj}{1+dj},\] which is in the kernel of the reduction map $\op{GL}_2(R)\rightarrow \op{GL}_2(R/J)$.
\par Note that $J^2=0$ since $J$ is contained in $\mathfrak{m}_R$ and $\mathfrak{m}_RJ=0$. Consider two pure tensors
\[\mtx{a_1}{b_1}{c_1}{d_1}\otimes j_1 \text{ and }\mtx{a_2}{b_2}{c_2}{d_2}\otimes j_2.\] From the relation $j_1j_2=0$, we have that
\[\left(\op{Id}+\mtx{a_1j_1}{b_1j_1}{c_1j_1}{d_1j_1}\right)\left(\op{Id}+\mtx{a_2j_2}{b_2j_2}{c_2j_2}{d_2j_2}\right)\]\[=\op{Id}+\mtx{a_1j_1+a_2j_2}{b_1j_1+b_2j_2}{c_1j_1+c_2j_2}{d_1j_1+d_2j_2}.\]Therefore we have an isomorphism (of groups)
\[\Ad \bar{\rho}\otimes_{\F_q} J\simeq \op{ker}\left\{\op{GL}_2(R)\rightarrow \op{GL}_2(R/J)\right\}.\]
\begin{Remark}\label{boringremark}The tensor product $\Ad \bar{\rho} \otimes_{\F_q} J$ is a $\op{G}_v$-module, via the adjoint action on $\Ad \bar{\rho}$ and the trivial action on $J$. Choosing an $\F_q$-basis of $J$, we find that as a $\op{G}_v$-module, the tensor product decomposes into a direct sum of copies of $\Ad \bar{\rho}$. As a result, $H^i(\op{G}_v, \Ad \bar{\rho} \otimes_{\F_q} J)$ is identified with $H^i(\op{G}_v, \Ad \bar{\rho}) \otimes_{\F_q} J$ for all $i\geq 0$.
\end{Remark}
\begin{Lemma}Let $R\rightarrow R/J$ be \textit{nearly small} and $v\in S$. The exponential map of a pure tensor \[X\otimes j\in H^1(\op{G}_v, \Ad\bar{\rho})\otimes_{\F_q} J\] takes $\varrho\in \Cv(R)$ to \[\exp(X\otimes j)\varrho:=(\operatorname{Id}+X\otimes j)\varrho.\]Refer to $\exp(X\otimes j)\varrho$ the \textit{twist} of $\varrho$ by $X\otimes j$. This gives a well defined action of $H^1(\G_v, \Ad\bar{\rho})\otimes_{\F_q} J$ on $\Cv(R)$ and the fibers of the reduction map \[\Cv(R)\rightarrow \Cv(R/J)\] are $H^1(\G_v, \Ad\bar{\rho})\otimes_{\F_q} J$ pseudotorsors.
\end{Lemma}
\begin{proof}
 Let \[X\otimes j\in H^1(\op{G}_v, \Ad\bar{\rho})\otimes_{\F_q} J\] be a pure tensor, $\varrho\in \Cv(R)$ and denote by $\varrho_0:\op{G}_v\rightarrow \op{GL}_2(R/J)$ the reduction of $\varrho$ modulo $J$. Let $\hat{X}$ be a cocycle representing the cohomology class $X$ and choose a representative $\hat{\varrho}$ for $\varrho$. Set \[\hat{\varrho}':\op{G}_v\rightarrow \op{GL}_2(R)\] to be equal to the twist
 \[\hat{\varrho}':=\left(\operatorname{Id}+\hat{X}\otimes j\right)\hat{\varrho}.\] Since $\hat{X}$ is a cocycle, it follows that $\hat{\varrho}'$ is a homomorphism. In greater detail, for $g_1, g_2\in \op{G}_v$, we have that
 \[\begin{split}
      \hat{\varrho}'(g_1)\hat{\varrho}'(g_2)= & \left(\operatorname{Id}+\hat{X}(g_1)\otimes j\right)\hat{\varrho}(g_1)\left(\operatorname{Id}+\hat{X}(g_2)\otimes j\right)\hat{\varrho}(g_2)\\
      = & \left(\operatorname{Id}+\hat{X}(g_1)\otimes j\right)\hat{\varrho}(g_1)\left(\operatorname{Id}+\hat{X}(g_2)\otimes j\right)\hat{\varrho}(g_1)^{-1}\hat{\varrho}(g_1g_2)\\
      =& \left(\operatorname{Id}+\hat{X}(g_1)\otimes j\right)\left(\operatorname{Id}+\left(\bar{\rho}_{\restriction v}(g_1)\hat{X}(g_2)\bar{\rho}_{\restriction v}(g_1)^{-1}\right)\otimes j\right)\hat{\varrho}(g_1g_2)\\
      =& \left(\operatorname{Id}+\hat{X}(g_1)\otimes j\right)\left(\operatorname{Id}+(g_1\hat{X})(g_2)\otimes j\right)\hat{\varrho}(g_1g_2)\\
      =& \left(\operatorname{Id}+\left(\hat{X}(g_1)+(g_1\hat{X})(g_2)\right)\otimes j\right)\hat{\varrho}(g_1g_2)\\
      =& \hat{\varrho}'(g_1g_2),
 \end{split}\] where $\bar{\rho}_{|v}$ is an abbreviation for $\bar{\rho}_{\restriction \op{G}_v}$. It is easy to see that if $\hat{X}$ is a coboundary, then $\hat{\varrho}'$ is strictly equivalent to $\hat{\varrho}$ (in the sense of Definition $\ref{stricteq}$). Therefore, this gives a well defined deformation 
 \[\varrho':=\left(\op{Id}+X\otimes j\right) \varrho\] of $\bar{\rho}_{\restriction \op{G}_v}$. Furthermore, $\varrho'$ is in fact a deformation of $\varrho_0$. This is because $\left(\op{Id}+X\otimes j\right)$ reduces to $\op{Id}$ and $\varrho$ reduces to $\varrho_0$ modulo $J$. Therefore, there is an action of $H^1(\op{G}_v, \Ad \bar{\rho})\otimes_{\F_q} J$ on the fibres of the reduction map \[\Cv(R)\rightarrow \Cv(R/J).\]
 \par We show that each fibre is an $H^1(\op{G}_v, \Ad \bar{\rho})\otimes_{\F_q} J$-pseudotorsor. This means that either the set of lifts $r\in \Cv(R)$ of $r_0\in \Cv(R)$ is empty, or the set of lifts is in bijection with $H^1(\op{G}_v, \Ad \bar{\rho})\otimes_{\F_q} J$. Assume without loss of generality that there exists a lift $r$ of $r_0$ and let $r'$ be any other lift. It suffices to show that there is a class $Y\in H^1(\op{G}_v, \Ad \bar{\rho})\otimes_{\F_q} J$ such that 
 \[r'=\left(\op{Id}+Y\right)r.\] It is easy to see that if such a class $Y$ exists then it is necessarily unique. Recall that the kernel of the reduction map 
 \[\op{GL}_2(R)\rightarrow \op{GL}_2(R/J)\]is identified with $\Ad\bar{\rho}\otimes_{\F_q} J$. For $g\in \op{G}_v$, set 
 \[\hat{Y}(g):=r'(g)r(g)^{-1}-\op{Id}\in \Ad\bar{\rho}\otimes_{\F_q} J,\] where we choose representatives for $r$ and $r'$.
 For $g_1, g_2\in \op{G}_v$, 
 \[\begin{split}r'(g_1)r'(g_2)=&\left(\op{Id}+\hat{Y}(g_1)\right)r(g_1)\left(\op{Id}+\hat{Y}(g_2)\right)r(g_2)\\ =& \left(\op{Id}+\hat{Y}(g_1)+g_1\hat{Y}(g_2)\right)r(g_1 g_2),
 \\ r'(g_1g_2)=&\left(\op{Id}+\hat{Y}(g_1g_2)\right)r(g_1 g_2),\end{split}\] and since $r'(g_1g_2)=r'(g_1)r'(g_2)$ it follows that 
 $\hat{Y}$ is a cocyle. We set $Y$ to be the associated cohomology class in $H^1(\op{G}_v, \Ad \bar{\rho} \otimes_{\F_q} J)$. Note $H^1(\op{G}_v, \Ad \bar{\rho} \otimes_{\F_q} J)$ is identified with $H^1(\op{G}_v, \Ad \bar{\rho}) \otimes_{\F_q} J$ (see Remark $\ref{boringremark}$). We have thus shown that there is a well defined class $Y$ such that $r'=(\op{Id}+Y)r$. This completes the proof.
\end{proof}

\begin{Remark}\label{Remarkonnk}
Let $R\in \mathfrak{C}$, \begin{enumerate}
\item\label{Remarkonnk1} $\mathfrak{n}_1=pR$ and that $\mathfrak{n}_k/\mathfrak{n}_{k+1}$ is an $R/\mathfrak{m}\simeq \F_q$ vector space.
\item\label{Remarkonnk2} If $R$ is a power series ring $R=\mathcal{O}[[U_1,\dots, U_s]]$ and $k\geq 1$, the quotient $\mathfrak{n}_k/\mathfrak{n}_{k+1}$ is an $\F_q$ vector space with basis representatives $[p^{k-b} U_1^{a_1}\dots U_s^{a_s}]\in \mathfrak{n}_k/\mathfrak{n}_{k+1}$ for $a_1,\dots, a_s,b\geq 0,b=a_1+\dots +a_s<k$, for instance, \[\mathfrak{n}_1/\mathfrak{n}_2=\F_q [p],\]
\[\mathfrak{n}_2/\mathfrak{n}_3=\F_q [p^2]\oplus \F_q [pU_1]\oplus\dots \oplus \F_q [pU_s].\]
\end{enumerate}
\end{Remark}

\begin{Lemma}\label{lemmadualSelmervanishing}
Let $X$ be the set of primes chosen as indicated at the start of this section, such that the Selmer and dual Selmer groups $H^1_{\mathcal{N}}(\G_{\Q,S\cup X}, \Ad^0\bar{\rho})$ and $H^1_{\mathcal{N}^{\perp}}(\G_{\Q,S\cup X}, \Ad^0\bar{\rho}^*)$ are both zero. Then
\begin{enumerate}
    \item $h^1_{\tilde{\mathcal{N}}}(\G_{\Q,S\cup X}, \Ad\bar{\rho})=1$
 and $h^1_{\tilde{\mathcal{N}}^{\perp}}(\G_{\Q,S\cup X}, \Ad\bar{\rho}^*)=0$,
    \item \[\Sh^2_{S\cup X}(\Ad\bar{\rho}):=\op{ker}\left( H^2(\op{G}_{\Q,S\cup X}, \Ad \bar{\rho})\rightarrow \bigoplus_{v\in S\cup X} H^2(\op{G}_v, \Ad \bar{\rho})\right)\] is equal to zero.
\end{enumerate}
\end{Lemma}
\begin{proof} We show that it follows that $h^1_{\tilde{\mathcal{N}}^{\perp}}(\G_{\Q,S\cup X}, \Ad\bar{\rho}^*)=0$. From this, it will follow from Corollary $\ref{SelmerdualSelmer1}$ that $h^1_{\tilde{\mathcal{N}}}(\G_{\Q,S\cup X}, \Ad\bar{\rho})=1$. Recall that $\Nv=H^1_{\operatorname{nr}}(G_v,\F_q)\oplus \mathcal{N}_v$ at all primes $v\in S\cup X\backslash\{p\}$ and $\tilde{\mathcal{N}}_p\supsetneq H^1_{\operatorname{nr}}(\G_p,\F_q)\oplus \mathcal{N}_p$. A class $f\in H^1_{\tilde{\mathcal{N}}^{\perp}}(\G_{\Q,S\cup X}, \Ad\bar{\rho}^*)$ is represented a sum of $f=f_1+f_2$ where 
\[f_1\in H^1(\G_{\Q,S\cup X}, (\F_q\cdot \operatorname{Id})^*)\text{ and 
}f_2\in H^1(\G_{\Q,S\cup X}, \Ad^0\bar{\rho}^*).\] We show that $f_2\in  H^1_{\mathcal{N}^{\perp}}(\G_{\Q,S\cup X}, \Ad^0\bar{\rho}^*)$. At each prime $v\in S\cup X$, the restriction of the class $f$ to $\G_v$ is perpendicular to $\mathcal{N}_v\subset \tilde{\mathcal{N}}_v$. Since $f_1$ takes values in $(\F_q\cdot \operatorname{Id})^*$ and $\mathcal{N}_v\subset \Ad^0\bar{\rho}$, $f_1$ on restriction to $\G_v$ is perpendicular to $\mathcal{N}_v$. Consequently, $f_2=f-f_1$ lies in $\mathcal{N}_v^{\perp}$ on restriction to $\G_v$. We deduce that $f_2$ lies in the dual Selmer group $H^1_{\mathcal{N}^{\perp}}(\G_{\Q,S\cup X}, \Ad^0\bar{\rho}^*)$ are therefore $f_2=0$.
\par Therefore, $f=f_1$. Since $\tilde{\mathcal{N}}_v$ contains the unramified classes $H^1_{\operatorname{nr}}(\G_v, \F_q\cdot \operatorname{Id})$, similar reasoning shows that $f_1$ is unramified everywhere. Let $\operatorname{Cl}(\Q(\mu_p))$ denote the class group of $\Q(\mu_p)$ with induced $\Gal(\Q(\mu_p)/\Q)$ action. It is a well known result (see \cite[Proposition 6.16]{washington}) that the $\bar{\chi}$ isotypic component of $(\operatorname{Cl}(\Q(\mu_p))\otimes \F_q)$ is equal to zero. An application of inflation-restriction shows that
\[H^1(\G_{\Q}, \F_q^*)\simeq H^1(\G_{\Q}, \F_q(\bar{\chi}))^{\Gal(\Q(\mu_p)/\Q)}\simeq \Hom(\G_{\Q(\mu_p)}, \F_q(\bar{\chi}))^{Gal(\Q(\mu_p)/\Q)}.\] We conclude that the unramified class $f_1=0$ and as a consequence, $f=0$.
\par Since $h^1_{\tilde{\mathcal{N}}^{\perp}}(\op{G}_{\Q,S\cup X}, \Ad\bar{\rho}^*)=0$ we deduce in particular that \[\Sh^1_{S\cup X}(\Ad\bar{\rho}^*):=\op{ker}\left( H^1(\op{G}_{\Q,S\cup X}, \Ad \bar{\rho}^*)\rightarrow \bigoplus_{v\in S\cup X} H^1(\op{G}_v, \Ad \bar{\rho}^*)\right)\] is equal to zero. By Global duality, $\Sh^2_{S\cup X}(\Ad\bar{\rho})\simeq \Sh^1_{S\cup X}(\Ad\bar{\rho}^*)^{\vee}=0$.
\end{proof}
For $R\in \Cfh$, let $\gamma_R:\mathcal{O}/p^3\rightarrow R/\mathfrak{n}_3$ and $\beta_R:\mathcal{O}/p\hookrightarrow R/\mathfrak{n}_1$ be induced from the structure map $\mathcal{O}\rightarrow R$. Before commencing with the proof of Theorem $\ref{main}$, let us briefly outline the strategy. Let $\mathcal{R}:=\mathcal{O}[[U]]\in \mathfrak{C}$ and $\mathfrak{m}=(p,U)$ its maximal ideal. The first step of the proof involves producing an appropriately chosen deformation
\begin{equation*}
\begin{tikzcd}[column sep= large]
 & \text{GL}_2(\mathcal{R}) \arrow[d]\\
\G_{\Q,S\cup X} \arrow[ru,"\tilde{\varrho}"] \arrow[r, "\bar{\rho}"] & \text{GL}_2(\mathcal{R}/\mathfrak{m}).
\end{tikzcd}
\end{equation*}
The second step involves showing that for $R\in \mathfrak{C}$, the induced map
\[\tilde{\varrho}^*:\Hom(\mathcal{O}[[U]],R)\rightarrow \D(R)\] is surjective.
\begin{proof}(of Theorem $\ref{main}$)
\par By Lemma $\ref{lemmadualSelmervanishing}$, we have that $h^1_{\tilde{\mathcal{N}}^{\perp},S\cup X}=0$ and therefore, from the Poitou-Tate long exact sequence, we get the short the short exact sequence
\begin{equation}\label{sesfinalproof}0\rightarrow H^1_{\tilde{\mathcal{N}}}(\G_{\Q,S\cup X},\Ad\bar{\varrho})\rightarrow H^1(\G_{\Q,S\cup X},\Ad\bar{\varrho})\rightarrow \bigoplus_{v\in S\cup X} \frac{H^1(\op{G}_v, \Ad\bar{\varrho})}{\Nv}\rightarrow 0.\end{equation} 
Pushing forward $\rho_3$ by the map $\gamma_{\mathcal{R}}$, we have the representation $\rho_{3,R}:\G_{\Q,S\cup X}\rightarrow \text{GL}_2(R/\mathfrak{n}_3)$. By Lemma $\ref{lemmadualSelmervanishing}$, the Selmer group $H^1_{\tilde{\mathcal{N}}}(\G_{\Q,S\cup X}, \Ad\bar{\varrho})$ is one-dimensional. Pick a generating element $g\in H^1_{\tilde{\mathcal{N}}}(\G_{\Q,S\cup X}, \Ad\bar{\varrho})$. As a vector space of $\mathcal{R}/\mathfrak{m}=\F_q$, we have that $\mathfrak{n}_2/\mathfrak{n}_3=\F_q [p^2]\oplus  \F_q [pU]$ (the notation is explained in Remark $\ref{Remarkonnk}$). Set
$\alpha:=g\otimes [pU]$ and set $\varrho_3=\exp(\alpha)\rho_3$. Since $\rho_3$ satisfies $\Phi$, so does $\varrho_3$.
Since $\Sh^2_{S\cup X}(\Ad^0\bar{\rho})$ and there are no local obstructions to lifting $\varrho_3$ to a mod $p^4$ representation, it follows that $\varrho_3$ lifts to
\[\rho_4:\G_{\Q,S\cup X}\rightarrow \GL_2(\mathcal{R}/\mathfrak{n}_4).\] From the surjectivity of the restriction map on the right of the the short exact sequence $\ref{sesfinalproof}$, it follows that there exists a cohomology class \[\beta^{(4)}\in H^1(\G_{\Q,S\cup X}, \Ad\bar{\rho})\otimes \mathfrak{n}_4/\mathfrak{n}_3\] such that the twist $\varrho_4:=\exp(\beta^{(4)})\rho_4$ satisfies $\Phi$. Repeating the same argument, it follows that since $\varrho_4$ satisfies $\Phi$, hence lifts to a representation $\varrho_5$ which also satisfies the local conditions $\Phi$.
In this fashion, a compatible system of deformations $\{\varrho_k\}_{k\geq 2}$ is constructed and the passage to the inverse limit of which yields a deformation \[\tilde{\varrho}:\G_{\Q,S\cup X}\rightarrow \text{GL}_2(\mathcal{R})\] satisfying $\Phi$, and which equals $\rho_{2,R}$ modulo $\mathfrak{n}_2(\mathcal{R})$. It is only at the mod $p^3$ level that the deformation is modified by $\alpha$.
\par Let $\sigma\in \D(\mathcal{S})$ and let $\mathfrak{m}_{\mathcal{S}}$ denote the maximal ideal of $\mathcal{S}$. Via a deformation theoretic argument, we construct a map $\zeta:\mathcal{R}\rightarrow \mathcal{S}$ of coefficient rings such that \[\sigma=\tilde{\varrho}^*(\zeta):\op{G}_{\Q}\xrightarrow{\tilde{\varrho}} \GL_2(\mathcal{R})\rightarrow \GL_2(\mathcal{S}).\] In the above, the latter map $\GL_2(\mathcal{R})\rightarrow \GL_2(\mathcal{S})$ is induced by $\zeta$. Set $\sigma_k:=\sigma\mod{\mathfrak{n}_k(\mathcal{S})}$. We let $\zeta_2: \mathcal{R}\rightarrow \mathcal{S}/\mathfrak{n}_2(S)$ be the unique $\mathcal{O}$-algebra map sending $U$ to $0$. Since $\tilde{\varrho}$ mod $\mathfrak{n}_2(\mathcal{R})$ is $\rho_{2,R}$, it follows that $\tilde{\varrho}^*(\zeta_2)$ is $\rho_{2,\mathcal{S}}$. As a result, $\sigma_2=\tilde{\varrho}^*(\zeta_2)$. We construct a compatible system of lifts $\{\zeta_k\}_{k\geq 2}$ such that for all $k$, we have that $\sigma_k=\tilde{\varrho}^*(\zeta_k)$. By the continuity of the operator $\tilde{\varrho}^*$, it shall follow that on passing to inverse limits, $\sigma=\tilde{\varrho}^*(\zeta)$, where $\zeta:=\varprojlim_k \zeta_k$. The map $\zeta$ constructed in this process need not be unique and the deformation $\varrho$ is not universal deformation representing $\D$. But this will show that the map $\eqref{surjective}$ is surjective. Further, we show that $\zeta$ is uniquely determined when $\mathcal{S}$ contains no $p$-torsion.
\par We proceed by induction on $k$. Suppose that for $k\geq 3$, $\zeta_{k-1}:\mathcal{R}\rightarrow \mathcal{S}/\mathfrak{n}_k(\mathcal{S})$ is constructed so that $\sigma_{k-1}=\tilde{\varrho}^*(\zeta_{k-1})$. We show that one may lift $\zeta_{k-1}$ to $\zeta_k$
\begin{equation}\label{noncommutative}
\begin{tikzcd}[column sep= large]
& \mathcal{S}/\mathfrak{n}_{k}(\mathcal{S})\arrow[d]\\
\mathcal{R} \arrow[ru,"\zeta_{k}"] \arrow[r, "\zeta_{k-1}"] & \mathcal{S}/\mathfrak{n}_{k-1}(\mathcal{S})
\end{tikzcd}
\end{equation}
so that $\sigma_k=\tilde{\varrho}^*(\zeta_k)$.
\par Being a formal power ring, $\mathcal{R}$ is formally smooth, and consequently, $\zeta_{k-1}$ lifts to $g_k:\mathcal{R}\rightarrow \mathcal{S}/\mathfrak{n}_k(\mathcal{S})$. Set
$\mu_k:=\tilde{\varrho}^*(g_k)$, both $\sigma_k$ and $\mu_k$ are deformations of $\sigma_{k-1}$ in $\D(\mathcal{S}/\mathfrak{n}_k(\mathcal{S}))$. There is a class \[\gamma\in H^1_{\mathcal{N}}(\G_{\Q,Z}, \Ad\bar{\rho})\otimes \mathfrak{n}_{k-1}(\mathcal{S})/\mathfrak{n}_k(\mathcal{S})\]
such that $\sigma_k=\exp(\gamma)\mu_k=(\operatorname{Id}+\gamma)\mu_k$. This class is simply $\sigma_k \mu_k^{-1}-\op{Id}$. Set $G:=g_k(U)$ and let $H\in \mathfrak{m}_{\mathcal{S}}$ be such that $\gamma$ may be represented as $\gamma= g\otimes [pH]$ with $[pH]\in \mathfrak{n}_{k-1}(\mathcal{S})/\mathfrak{n}_{k}(\mathcal{S})$. Note that the choice of $H$ need not be unique when $\mathcal{S}$ has nontrivial $p$-torsion. Note that $pH\in \mathfrak{m}_{\mathcal{S}}^2$ since $k-1\geq 2$ and $p\notin \mathfrak{m}_{\mathcal{S}}^2$ since $\mathcal{S}\in \Cfh$. Therefore $H$ is not a unit in the local ring $\mathcal{S}$  and thus in the maximal ideal $\mathfrak{m}_{\mathcal{S}}$.
\par Let $\zeta_k:\mathcal{R}\rightarrow \mathcal{S}/\mathfrak{n}_k(\mathcal{S})$ be the $\mathcal{O}$-algebra map which takes
\[U\mapsto G+H.\]
\par We will now show that the effect of replacing $g_k$ by $\zeta_k$ is that $\mu_k=g_k^*(\tilde{\varrho})$ gets replaced by $\sigma_k=\exp(\gamma)\mu_k$. This will conclude the proof.
\par First we observe that $\zeta_k(x)=g_k(x)$ for $x\in \mathfrak{n}_3(\mathcal{R})$.
The ideal $\mathfrak{n}_3(\mathcal{R})$ is generated by monomials $p^3, p^2U, pU^2$. Since $pH\in \mathfrak{n}_{k-1}(\mathcal{S})$, we have that $pHz\in \mathfrak{n}_k(\mathcal{S})$ for any $z\in \mathfrak{m}_{\mathcal{S}}$. Consequently, an application of the binomial theorem yields that for a monomial $p^aU^{b}$ for which $a+b=3$ and $a\geq 1$,
\[\zeta_k(p^aU^{b})=p^a(G+H)^{b}
=p^aG^{b}
=g_k(p^aU^{b}).\]
Hence $\zeta_k(x)=g_k(x)$ for $x\in \mathfrak{n}_3(\mathcal{R})$.
\par The maps $\zeta_k$ and $g_k$ induce maps $\op{M}_2(\zeta_k)$ and $\op{M}_2(g_k)$ on the $2\times 2$ matrices $\op{M}_2(\mathcal{R})\rightarrow \op{M}_2(\mathcal{S}/\mathfrak{n}_k(\mathcal{S}))$. Recall from the construction of $\tilde{\varrho}$ that in matrix notation
\begin{equation}\label{varrhomod3R}\tilde{\varrho}=\op{exp}(\alpha)\rho_3=(\operatorname{Id}+gpU)\rho_3=\rho_3+gpU \bar{\rho}  \mod{\mathfrak{n}_3(\mathcal{R}).}\end{equation}
From the relation, $\ref{varrhomod3R}$ we deduce that as a function to $2\times 2$ matrices $\op{M}_2(\mathcal{S}/\mathfrak{n}_k(\mathcal{S}))$
\[\tilde{\varrho}^*(\zeta_k)-\mu_k:\G_{\Q,S\cup X}\rightarrow \op{M}_2(\mathcal{S}/\mathfrak{n}_k(\mathcal{S}))\] evaluates to
\[\begin{split}\tilde{\varrho}^*(\zeta_k)-\mu_k&=\tilde{\varrho}^*(\zeta_k)-\tilde{\varrho}^*(g_k)\\
&=\op{M}_2(\zeta_k)\circ\left(\rho_3+g \bar{\rho} pU+\mathfrak{n}_3(\mathcal{R})\right)-\op{M}_2(g_k)\circ \left(\rho_3+ g \bar{\rho} pU+\mathfrak{n}_3(\mathcal{R})\right). 
\end{split}
\]Since $\phi_k(x)=g_k(x)$ for $x\in \mathfrak{n}_3(\mathcal{R})$, the above may be represented as
\[\begin{split}
&=\left(\rho_3+ g \bar{\rho} p\zeta_k(U)\right)- \left(\rho_3+g \bar{\rho} pg_k(U)\right)\\
&=gpH\bar{\rho}=\gamma \bar{\rho}.
\end{split}\]
Consequently
\[\tilde{\varrho}^*(\zeta_k)=\mu_k+\gamma\bar{\rho}=\mu_k+\gamma\mu_k=(\operatorname{Id}+\gamma)\mu_k=\sigma_k.\] This completes the induction step. Therefore,  the map $\eqref{surjective}$ is surjective. We observe that when $\mathcal{S}$ has no $p$-torsion the element $H$ is uniquely determined. Therefore, it also follows by induction that there is a unique $\zeta$ such that $\tilde{\varrho}^*(\zeta)=\sigma$. This completes the proof of the theorem.\end{proof}

\begin{proof}(of Theorem \ref{aux}) Let $R\in \mathfrak{C}$ and $\mathfrak{f}\in \Hom_{\mathfrak{C}}(\mathcal{O}[[U]], R)$ with associated deformation $\rho_{\mathfrak{f}}\in \D(R)$. We take note that the weight of $\rho_{\mathfrak{f}}$ is defined as the composite \[\mathscr{W}t^*\rho_{\mathfrak{f}}:\Lambda\xrightarrow{\mathscr{W}t}\mathcal{O}[[U]]\xrightarrow{\mathfrak{f}}R\] and that $\mathscr{W}t(x)=\det \tilde{\varrho}-1$. Let $\Omega(R)\subseteq \Hom_{\mathfrak{C}}(\Lambda,R)$ be the subset of weights which are congruent to the weight of the prescribed deformation \[\Omega(R):=\{\lambda\in\Hom_{\mathfrak{C}}(\Lambda,R)\mid \lambda\equiv \mathscr{W}t^*\rho_2\mod(pR\cap \mathfrak{m}_R^2)\}.\]
We observe that since $\mathscr{W}t\in \Omega(\mathcal{O}[[U]])$, it follows that $\mathscr{W}t^*\rho_{\mathfrak{f}}\in \Omega(R)$. Let $\lambda\in \Omega(R)$, we show that there exists a deformation $\varrho^{\lambda}\in \D(R)$ with weight $\mathscr{W}t^*\varrho^{\lambda}=\lambda$. We refer to the proof of Theorem $\ref{main}$ to the choice of the set of primes $X$, these are chosen so that the dual Selmer group of the fixed weight Selmer conditions is zero. We fix the weight $\lambda$ and examine if $\rho_2$ has a deformation to $\D(R)$. Standard techniques in deformation theory discussed in this manuscript (cf. \cite{hamblenramakrishna}) imply that a lift does indeed exist provided the (fixed weight) dual Selmer group satisfies
\[H^1_{\mathcal{N}^{\perp}}(\G_{\Q,S\cup X}, \Ad^0\bar{\rho}^*)=0.\]
This concludes the proof of the Theorem.
\end{proof}

\end{document}